\documentclass[11pt]{article}

\usepackage{amsmath, amsthm, amsfonts, amssymb}
\newtheorem{theorem}{Theorem}[section]

\newtheorem{lemma}[theorem]{Lemma}
\newtheorem{proposition}[theorem]{Proposition}
\newtheorem{example}[theorem]{Example}
\newtheorem{remarks}[theorem]{Remark}
\newtheorem{remark}[theorem]{Remark}
\newtheorem{definition}[theorem]{Definition}

\numberwithin{equation}{section}

\title{\bf Positive curvature property for sub-Laplace on nilpotent Lie group of rank two  
}
\author{Bin Qian  \thanks{ Department of Mathematics, Changshu Institute of Technology, Changshu, Jiangsu 215500,  China, and School of Mathematical Sciences, Fudan University,
220 Handan Road, Shanghai 200433.
 E-mail: binqiancn@yahoo.com.cn, binqiancn@gmail.com}
 }\date{}

\newcommand{\rr}{\mathbb{R}}

\def\LL{\mathcal L}

\def\<{\langle}
\def\>{\rangle}

\def\bequ{\begin{equation}}
\def\nequ{\end{equation}}

\def\bdef{\begin{defn}}
\def\ndef{\end{defn}}

\def\bthm{\begin{thm}}
\def\nthm{\end{thm}}

\def\bprop{\begin{prop}}
\def\nprop{\end{prop}}

\def\brmk{\begin{remarks}}
\def\nrmk{\end{remarks}}

\def\bexam{\begin{example}}
\def\nexam{\end{example}}

\def\blem{\begin{lemma}}
\def\nlem{\end{lemma}}

\def\bcor{\begin{cor}}
\def\ncor{\end{cor}}

\def\bprf{\begin{proof}}
\def\nprf{\end{proof}}

\def\bdes{\begin{description}}
\def\ndes{\end{description}}

\def\part{\partial}

\begin{document}
\maketitle
\begin{abstract}
In this note,  we concentrate on the sub-Laplace on the nilpotent Lie group of rank two, which is the infinitesimal generator of the diffusion generated by $n$ Brownian motions and their $\frac{n(n-1)}2$ L\'evy area processes, which is the simple extension of the sub-Laplace  on the Heisenberg group $\mathbb{H}$. In order to study contraction properties of the heat kernel, we show that, as in the cases of the  Heisenberg group
 and  the three Brownian motion model, the restriction of the sub-Laplace  acting
on radial functions (see Definition \ref{radial fun}) satisfies a  positive Ricci curvature
condition (more precisely a  $CD(0, \infty)$ inequality, see Theorem \ref{positive}, whereas the operator itself does not  satisfy any $CD(r,\infty)$ inequality. From this we may deduce some useful, sharp gradient bounds for the associated heat kernel. It can be seen a generalization of the paper \cite{Qian2}.

\end{abstract}


\vskip10pt \textbf{2000 MR Subject Classification:} 60J60 58J35

\section{Introduction }

\subsection*{ The elliptic case}
 Let $M$ be a complete Riemannian manifold of dimension
$n$ and let $\LL:=\Delta+\nabla h$, where $\Delta$ is the
Laplace-Beltrami operator. For $t\ge 0$, denote by $P_t$ the heat
semigroup generated by $\LL$ (that is formally $P_{t}= \exp(t\LL)$).
For smooth enough function $f,g$,  one defines (see \cite{Ba97})
$$\aligned
\Gamma(f,g)&=|\nabla f|^2=\frac12(\LL fg-f\LL g-g\LL f),\\
 \Gamma_2(f,f)&=\frac12\big(\LL\Gamma(f,f)-2\Gamma(f,\LL
f)\big)=|\nabla\nabla f|^2+(Ric-\nabla\nabla h)(\nabla f,\nabla
f).\endaligned$$ We have the following well-known proposition, see
Proposition 3.3 in \cite{Ba97}.

\noindent{\bf Proposition A.} For every real $\rho\in \rr$, the
following are equivalent \bdes
\item{(i).} $CD(\rho,\infty)$ holds. That is
$\Gamma_2(f,f)\ge \rho\Gamma(f,f)$.

\item{(ii).} For $t\ge0$, $\Gamma(P_tf,P_tf)\le e^{-2\rho
t}P_t(\Gamma(f,f)).$

\item{(iii).} For $t\ge 0$, $\Gamma(P_tf,P_tf)^{\frac12}\le e^{-\rho
t}P_t(\Gamma(f,f)^{\frac12}).$

\ndes
Moreover, in \cite{Eng06}, Engoulatov obtained the following
gradient estimates for the associated heat kernels $p(t,x,y)$ in Riemannian manifolds.

 \noindent {\bf Theorem B. } Let $M$ be a complete Riemannian of
dimension $n$ with Ricci curvature bounded from below, $Ric(M)\ge
-\rho$, $\rho\ge0$. \bdes
\item{(i).} Suppose a non-collapsing condition is satisfies on $M$,
namely, there exist $t_0>0$, and $\nu_0>0$, such that for any $x\in
M$, the volume of the geodesic ball of radius $t_0$ centered at $x$
is not too small, $ Vol(B_x(t_0))\ge \nu_0.$ Then there exist two
constants $C(\rho,n,\nu_0,t_0)$ and $\bar{C}(t_0)>0$, such that
$$
|\nabla \log p(t,x,y)|\le
C(\rho,n,\nu_0,t_0)\left(\frac{d(x,y)}{t}+\frac1{\sqrt{t}}\right),
$$uniformly on $(0, \bar{C}(t_0)]\times M\times M$, where $d(x,y)$
is  the Riemannian distance between $x$ and $y$.

\item{(ii).} Suppose that $M$ has a diameter bounded by $D$, Then
there exists a constant $C(\rho, n)$ such that
$$
|\nabla\log p(t,x,y)|\le
C(\rho,n)\left(\frac{D}{t}+\frac1{\sqrt{t}}+\rho\sqrt{t}\right),
$$
uniformly on $(0,\infty)\times M\times M$.

 \ndes
 Recently, X. D. Li \cite{XLi} has shown that the non-collapsing condition can be removed.
\subsection*{The hypoelliptic case}
 More recent, some focus has been set on some degenerate (hypoelliptic) situations, where the methods used for the elliptic case do not apply. Among the simplest examples of such situation is the  Heisenberg group, denote $p_t$  the heat
kernel of  Markov semigroup $P_t$ at the origin $o$ with respect to Lebesgue measures on $\rr^3$, see \cite{Gaveau,HQLi,HLi1} for  the precise formulas.  H. Q. Li obtain the sharp gradient estimate  for the heat kernel $p_t$ and the contraction property for the semigroup $P_t$, which generalizes and strengthens the result of Driver and Melcher, \cite{DM}.

\noindent {\bf Theorem C. } For any $g\in \mathbb{H}$, we have
\begin{equation}\label{Hgradient} |\nabla \log
p_t|(g)\le \frac{Cd(g)}{t}, \end{equation}
where $d(g)$ is the Carnot-Carath\'eodory distance between $o$ and $g$. In addition, we have
  \begin{equation}\label{heisenberg-Li} \forall f\in
C_0^\infty(\mathbb{H}), \ \forall t\ge 0,\
\Gamma(P_tf,P_tf)^{\frac12}\le
C_1P_t\big(\Gamma(f,f)^{\frac12}\big). \end{equation}  (See also  D. Bakry et
al. \cite{BBBC} for alternate proofs.)

  The method adopted relies intensely on the precise  asymptotic estimates for the heat kernel. In the similar way, H. Q. Li and his collaborator in \cite{HLi1,HLi2,HL},   show that (\ref{Hgradient}) and (\ref{heisenberg-Li}) hold in the Heisenberg type group $H(2n,m)$, see also \cite{HLi3} for the Grushin operators. For $SU(2)$ group,  F. Baudoin and M. Bonnefont show that a modified form of
(\ref{Hgradient}) and (\ref{heisenberg-Li}) hold in \cite{BB}.   The author himself shows that the gradient estimate (\ref{Hgradient}) holds  for  the three Brownian motion model in \cite{Qian2}, see also \cite{Qian1} for the high dimensional Heisenberg group.

  In this note, we  shall focus on  the  nilpotent Lie group of rank two (It can also be called the $n$-Brownian motion model),  which can be seen an another typical simpe example of hypoelliptic operator, but  the structure is more complex than the Heisenberg (type) groups. Up to the author's knowledge, the method of H.Q. Li, \cite{HL}-\cite{HLi3}, fails  to study the precise gradient bounds in this context.

As the three Brownian motion model \cite{Qian2}, we shall first look at the symmetries, that is we shall characterize all the vector fields  which commute with the
sub-Laplace $\Delta$, see Proposition \ref{linear-prop}. The
infinitesimal rotations are those vector fields which vanish at the orgin $o$
and a radial function is a function which vanishes on infinitesimal
rotations. In this case, although the Ricci curvature is everywhere
$-\infty$, refer to \cite{Juillet,BBBC}, we shall prove that the
$\Gamma_2$ curvature is  positive along the radial directions,
as it is the case for the Heisenberg group and three Brownian motions model, see Theorem \ref{positive}. The difficulty  for general $n$($n>3$) is that it is not easy to prove the positive curvature property directly even in the case of $4$ Brownians motion model, since it is not easy to get the explicit, well organized solutions to the linear equations as the ones in the Proposition 3.1 in \cite{Qian2}. Even it is getting more and more complex as $n$ grows. Inspired by the work of  T. Melcher, c.f. \cite{thesis}, we will  firstly prove $L^1$ heat kernel inequality for radial functions (see definition \ref{radial fun}), and hence the positive property of Bakry-Emery $\Gamma_2$ curvature holds along the radial directions. As a consequence, the same form of gradient estimate (\ref{Hgradient}) holds  by combining the method developed by F. Baudoin and M. Bonnefont in
\cite{BB} with the method in \cite{HLi1}. It is worth
recalling that in \cite{BBBQ}, D. Bakry et al. have obtained the
Li-Yau type gradient estimates for the three dimensional model group
by applying $\Gamma_2$-techniques which plays an essential role in the paper. In our setting, it is easy to see
that this type of gradient estimate  also holds.

\section{Nilpotent Lie group of rank two--$n-$dimensional Brownian motion model $N_{n,2}$}
Let us recall the definition  of nilpotent Lie group of rank two, see \cite{Gaveau,VSC}.
\begin{definition}  A linear space $\frak{g}$ is a nilpotent Lie group of rank two if $\frak{g}=V_1\oplus V_2$, where $V_1,V_2$ are vector subspace of $\frak{g}$, satisfying $V_2=V_1\oplus V_1,[V_1,V_2]=0$ and $[V_2,V_2]=0$. We denote $\frak{N}_{n,2}$ the nilpotent algebra with $n$ generators, denote $N_{n,2}$ the simple connected Lie group of rank two with the algebra $\frak{N}_{n,2}$.
\end{definition}
Suppose $V_1$ is spanned by  $X_i, 1\le i\le n$ and  $V_2$ is generated by $Y_{ij}:=[X_i, X_j], i<j$. In this case,  we have $[X_k, Y_{ij}]=0$, $1\le i<j\le n, 1\le k\le n$. The nature sub-Laplace operator
 is defined by \bequ\label{sublaplacian} \Delta=\sum_{i=1}^nX_i^2. \nequ

   Under the certain exponential map on $N_{n,2}$,  without loss of any generality,  we can assume $X_i,Y_{ik}$ has the following form, see Lemma
4.1 in \cite{Gaveau},    \bequ\label{expression}
\begin{cases}
X_i&=\part_i+\frac12\left(\sum_{k<i}x_k\hat{\part}_{ki}-\sum_{k>i}x_k\hat{\part}_{ik}\right),\\
Y_{ik}&=\hat{\partial}_{ik},\end{cases} \nequ
for $1\le i,k\le n $, with the notation $\part_i=\frac{\part}{\part x_i},\
\hat{\part}_{ik}=\frac{\part}{\part y_{ik}}$. The reason why we call it the $n$ Brownian motions model is that
$\frac12\Delta$ is the infinitesimal generator of the Markov process
$\big(\{B_i\}_{1\le i\le
n},\{\frac12\int_0^tB_idB_{i+1}-B_{i+1}dB_{i}\}_{1\le i\le n}\big)$,
where $\{B_i\}_{1\le i\le n}$ are $n$  real standard independent
Brownian motions.

By convention,  for all $t\ge0$, denote $P_t:=e^{t\Delta}$  the associated heat
semigroup generated by the canonical sub-Laplacian  $\Delta$, $p_t$ the heat kernel of $P_t$ at the origin  $o$
with respect to the Lebesgue measure on $\rr^{\frac{n(n+1)}{2}}$.

For any function $f,g$ defined on $N_{n,2}$, the carr\'e du champ operators are, see \cite{Ba97,Ledoux},
\begin{align*}
\Gamma(f,g)&:=\frac12(\Delta(fg)-f\Delta g-g\Delta f)\\
&=
\sum_{i=1}^nX_ifX_ig,
\end{align*}
and
\begin{align*}
 \Gamma_2(f,f)&:=\frac12(\Delta\Gamma(f,f)-2\Gamma(f,\Delta f))\\
 &=\sum_{i,j}(X_iX_jf)^2+2\sum_{i<j}X_jfX_iY_{ij}f-X_ifX_jY_{ij}f.
\end{align*}
 Here the mixed term $\sum_{i<j}X_jfX_iY_{ij}f-X_ifX_jY_{ij}f$ prevent the existence of any constant $\rho$ such that the curvature dimensional condition $CD(\rho,\infty)$ holds, see \cite{Juillet}. Nevertheless, we have the following Driver-Melcher inequality, see \cite{DM,M08},
$$
\Gamma(P_tf,P_tf)\le CP_t\Gamma(f,f).
$$
for some positive constant $C$. The constant $C$ here can be
expressed explicitly following  the method in \cite{BBBC} by dilation equation. For the Bakry-Emery heat kernel inequality  (\ref{heisenberg-Li}), the methods deeply rely on the precise estimate on the heat kernel $p_t$ and its differentials (see \cite{HQLi,BBBC,HL}). Up to the author's knowledge, these precise estimates are not known for the model $N_{n,2}$, neither the heat kernl inequality (\ref{heisenberg-Li}) (or so-called H. Q. Li inequality). Nevertheless, we shall prove that one of the key gradient estimates (\ref{Hgradient}) holds, which would be a first step for the proof of the H. Q. Li inequality in this context. We remark that it does hold for radial functions, see Proposition \ref{L1}, see also \cite{thesis} for some other function classes.

For the heat kernel $p_t$, we have the following property, for $\vec{x}\in\rr^{n},\vec{y}\in \rr^{\frac{n(n-1)}{2}}$, see \cite{Gaveau}, \begin{equation}\label{scale}
p_t(\vec{x},\vec{y})=t^{-n^2/2}p_1(\vec{x}/\sqrt{t},\vec{y}/t),
\end{equation}
hence it is enough to study the heat kernel $p_t$ at time $t=1$. For $t=1$, we have for $\vec{x}=(x_1,\cdots,x_n)^t\in\rr^n,\ \vec{y}=(y_{12},\cdots,y_{n-1,n})^t\in\rr^{\frac{n(n-1)}2}$, see P. 125, Theorem
 1 in  \cite{Gaveau},
\begin{equation}\label{express0} p(\vec{x},\vec{y}):=p_1(\vec{x},\vec{y})=(2\pi)^{-\frac{n(n+2)}2}\int_{\rr^{\frac{n(n-1)}2}}\exp{\left(-i
\sum_{k<l}\alpha_{kl} y_{kl}\right)}\prod_{j=1}^{[\frac{n}2]}\varphi_j(A,\vec{x})\prod_{k<l}d\alpha_{kl},
\end{equation}
where
$$
\varphi_j(A,\vec{x})=\frac{P_{2j-1}}{2}\left(\sinh \frac{P_{2j-1}}{2}\right)^{-1}\exp\left(-\frac{(\Omega^t\vec{x})_{2j-1}^2+(\Omega^t\vec{x})_{2j}^2}{2}
\frac{P_{2j-1}}{2}\coth \frac{P_{2j-1}}{2}\right),
$$
with $A$ is an antisymmetric matrix with  the entries $\{\alpha_{kl}\}_{k<l}$ in the upper triangular and $\Omega$ is the orthogonal matrix satisfying $\Omega^tA\Omega=P$, where $P$ is a antisymmetric matrix formed by diagonal block of
$$\left(
\begin{matrix}
0&P_{2k-1}\\
-P_{2k-1}&0
\end{matrix}\right),\ \ 1\le k\le \frac{n}2,\  \mbox{if}\ \mbox{n is even},
$$
and if $k$ is odd, $1\le k\le [\frac{n}2]$, the last block  is $1\times1$ zero matrix, where $iP_{2j-1}\  (P_{2j-1}\in\rr^+)$ is the eigenvalue of the antisymmetric matrix $A$. Without loss of any generality, we can assume
$P_1\ge P_3\ge \cdots\ge P_{2[\frac{n}2]}-1$.

The natural distance, induced by the sub-Laplace $\Delta$, is
the Carnot-Carath\'eodory distance $d$.  As usual, it can be defined
from the gradient operator $\Gamma$ only  by, see \cite{Ba97,VSC},
\bequ\label{ccdistance}
d(g_1,g_2):=\sup_{\{f:\Gamma(f)\le1\}}f(g_1)-f(g_2). \nequ For this
distance, we have the invariant and scaling properties, see
\cite{Gaveau,VSC}.
$$d(g_1,g_2)=d(g_2^{-1}\circ g_1, o):=d(g_2^{-1}\circ g_1),\
\mbox{and}\  d(\gamma \vec{x},\gamma^2\vec{y})=\gamma d(\vec{x},\vec{y}),$$ for
all $g_1,g_2\in {N}_{n,2}$, $\gamma\in\rr^+$ and
$\vec{x}\in \rr^n, \vec{y}\in\rr^{\frac{n(n-1)}2}$.

\section{Radial functions}
 In this section, we will give the precise definition of radial functions. To this end, we study the rotation vectors in $N_{n,2}$.

 Denote
$\mathcal{T}$ be the linear space for such vectors (spanned by the vectors $X_i,Y_j$, $1\le i,jle n$), which commute to the sub-Laplace
$\Delta$. Now we trivially know that for $i<k$, $Y_{ik}$ commutes to
$\Delta$ since $Y_{ik}$ commutes to $X_i$. Actually, there are lots
of vectors who share this property. For simplicities, for $1\le i<j\le n,$, denote
\begin{equation}\label{rotation}
\theta_{ij}=x_j\part_i-x_i\part_j+\sum_{1\le
k<i}y_{kj}\hat{\part}_{ki}-y_{ki}\hat{\part}_{kj}
+\sum_{i<k<j}y_{ik}\hat{\part}_{kj}-y_{kj}\hat{\part}_{ik}
 +\sum_{j<k\le n}y_{jk}\hat{\part}_{ik}-y_{ik}\hat{\part}_{jk},
\end{equation}
and
$$\hat{X}_{i}=\part_i-\frac12\left(\sum_{k<i}x_k\hat{\part}_{ki}-\sum_{k>i}x_k\hat{\part}_{ik}\right).$$
In fact, $\{X_i\}_{1\le i\le n}$ ($\{\hat{X}_i\}_{1\le i\le n}$) can be called the left (right) invariant vectors respectively. For the vectors $\theta_{ij}$, we have the following Lie relations: for $1\le i<j<k\le n$,
 \bequ\label{rotation-Lie}
[\theta_{ij},\theta_{ik}]=\theta_{jk}.
 \nequ

Let us state the main result in this section.

 \begin{proposition}\label{linear-prop}
$$ \mathcal{T}=Linear\{\hat{X}_i, Y_i, \theta_{ij}, 1\le
i,j\le n\}.
$$Here $Linear$ means the linear combination of vectors, with the constant
coefficients.

 \end{proposition}

\begin{remarks}
 In particularly, for $n=2$, this case is called the Heisenberg
group, we can actually induce a group act such that $X_i$ are
corresponding to the left vector fields, see \cite{Gaveau,BGG,BBBC}.
In this case we have $\mbox{dim} \mathcal{T}=4$, and
$\mathcal{T}=\mbox{Span} \{\hat{X}_1,\hat{X}_2, Y, \theta\}$, where
$$\hat{X}_1=\part_{x_1}+\frac{x_2}{2}\part_y,\
 \hat{X}_2=\part_{x_2}-\frac{x_1}{2}\part_y,\
 \theta=x_1\part_{x_2}-x_2\part_{x_1}.$$
 Here  "Span" means the linear combination with the constant
 functions.

 For $n=3$. Actually, with changing the sign, we can also introduce
a group act such that $X_i$ are corresponding the left vector
fields. Explicitly, see \cite{Gaveau}$$ \aligned
X_1&=\part_1-\frac{x_2}2Y_3+\frac{x_3}2Y_2,\ \ \hat{X}_1=\part_1+\frac{x_2}2Y_3-\frac{x_3}2Y_2;\\
X_2&=\part_2-\frac{x_3}2Y_1+\frac{x_1}2Y_3,\ \ \hat{X}_1=\part_1+\frac{x_2}2Y_3-\frac{x_3}2Y_2;\\
X_3&=\part_3-\frac{x_1}2Y_2+\frac{x_2}2Y_1,\ \
\hat{X}_3=\part_3+\frac{x_1}2Y_2-\frac{x_2}2Y_1,
\endaligned
$$
where $\hat{X_i}$ are the right vector fields,
$Y_i=\part_{y_i}:=\hat{\part}_i$. In this case, we have $$ \mathcal{T}=\mbox{Linear}\{\hat{X}_i,
Y_{i},\theta_i, 1\le i\le 3\},
$$
where$$ \aligned
\theta_1&=x_2\part_3-x_3\part_2+y_2\hat{\part}_3-y_3\hat{\part}_2,\\
\theta_2&=x_3\part_1-x_1\part_3+y_3\hat{\part}_1-y_1\hat{\part}_3,\\
\theta_3&=x_1\part_2-x_2\part_1+y_1\hat{\part}_2-y_2\hat{\part}_1.
\endaligned
$$
It has been shown in \cite{Qian2}.
\end{remarks}


 To proof this Proposition, suppose any vector
$X=\sum_ia_iX_i+\sum_{i<j}b_{ij}Y_{ij}$,
 satisfying $[\Delta, X]=0$, where $a_i, b_{ij}$ are the functions
 in $\{x_{\cdot},y_{\cdot\cdot}\}$. Denote
 $W_{ij}=X_iX_j+X_jX_i$, then $X_iX_j=\frac12(W_{ij}+Y_{ij})$. Note
 that
$$
\aligned
\hskip 2pt [\Delta, X] &=\sum_{i,j}X_i^2a_jX_j+2X_ia_jX_iX_j+2a_jX_iY_{ij}\\
&\hskip 24pt +\sum_{i<j,k} X_k^2b_{ij}Y_{ij}+2X_kb_{ij}X_kY_{ij}\\
&=\sum_{i,j}X_i^2a_jX_j+\sum_{i<j}(X_ia_j+X_ja_i)W_{ij}+2X_ia_iX_i^2
+\sum_{i<j}(X_ia_j-X_ja_i+\sum_kX_k^2b_{ij})Y_{ij}\\
&\hskip 24pt+2\sum_{i<j}(a_jX_iY_{ij}-a_iX_jY_{ij})+\sum_{k\neq
i,j,i<j
}2X_kb_{ij}X_kY_{ij}+2\sum_{i<j}X_ib_{ij}X_iY_{ij}+2\sum_{i<j}X_jb_{ij}X_jY_{ij}
\endaligned
$$
thus we have \begin{eqnarray} &\ &\sum_iX_i^2a_j=0,\label{equ-zero0}\\
X_ia_j=-X_ja_i,& \ &X_ia_i=0,\label{equ-zero1}\\
\sum_kX_k^2b_{ij}=2X_ja_i, \ i<j,&\
&X_kb_{ij}=0, k\neq i,j, \ i<j,\label{equ-zero2}\\
X_ib_{ij}=-a_j,\ i<j,&\  &X_jb_{ij}=a_i, i<j.\label{equ-zero3}
\end{eqnarray}

\blem\label{linear-lem1} $a_i, 1\le i\le n$ are linear functions in
$\{x_i,1\le i\le n\}$, they are independent on $\{y_{ik},i<k\}$.
\nlem
 \bprf \bdes
\item{\bf Step1:} For fixed $i$, $j>i$, $[X_i,Y_{ij}]=0$,
 Combining (\ref{equ-zero2}),
 we have for $i\neq k,l, k<l$,
 $X_iY_{ij}b_{kl}=Y_{ij}X_ib_{kl}=0$. Since $Y_{ij}=[X_i,X_j]$,
 again using the fact (\ref{equ-zero2}), we have
 $X_i^2X_jb_{kl}=0$, $i\neq k,l, k<l, i<j$. By choosing $l=j$, and
 using the fact (\ref{equ-zero3}), we have $X_i^2a_k=0, k<j,k\neq
 i$. In the same way, we have $X_i^2a_l=0, \ i<j<l. $ Combining
 (\ref{equ-zero1}), we have
 \bequ\label{square-two1}X_i^2a_j=0,\
  \mbox{for}\  1\le i,j\le n.
  \nequ

\item{\bf Step2:} Again for $i<j,l$, $[X_i, Y_{ij}]b_{il}=0$. By the fact that $[X_i,X_j]=Y_{ij}$ and (\ref{equ-zero3}), we have
$X_i^2X_jb_{il}+X_iX_ja_l=X_jX_ia_l-X_iX_ja_l$. Again we use the
fact that  $X_i^2X_jb_{il}=0$, which has been proved above, we get
$$2X_iX_ja_l=X_jX_ia_l, \ \mbox{for}\ i<j,l.$$ For $i<j<l$, start
from the fact that $[X_j,Y_{ij}]b_{jl}$=0,  we have
$$
2X_jX_ia_l=X_iX_ja_l, \ i<j<l.
$$
Combining the above two equations, and (\ref{square-two1}) we have
$$ X_iX_ja_l=X_jX_ia_l=0,\  i\le j\le l. $$ Using $X_ia_j=-X_ja_i$,
we have \bequ\label{square-two} X_iX_ja_l=0, \ 1\le i,j,l\le n.
\nequ
\item{\bf Step3:} By the fact that $Y_{ij}=[X_i,X_j]$, with
(\ref{square-two}), we have $Y_{ij}a_l=0$, for $1\le l\le n,\ i<j$.
Thus $\{a_k\}_{1\le k\le n}$ is independent on $\{y_{ik}\}_{1\le
i<k\le n}$, that is $a_k$ is the function in $\{x_i\}_{1\le i\le
n}$. From the definition of $X_k$, we have $X_ka_j=\part_{k}a_j$. By
(\ref{square-two}), i.e. for $1\le i,k\le n$, $\part^2_ia_k=0$, thus
we can conclude $a_k$ in linear function in $x_i$.
 \ndes
\nprf Thus we can give the explicit expression for $a_i$, for $1\le i\le
n$,  \bequ\label{expression-a} a_i=\sum_{j=1}^nA_{ij}x_j+B_i,\nequ
where $A_{ij}, B_i$ are constants and $A_{ij}$ satisfies
$A_{ij}=-A_{ji}$.

Note that we can write
\bequ\label{expression-linear}X=\sum_{i=1}^na_i\part_i+\sum_{i<j}c_{ij}\hat{\partial}_{ij},\nequ
with $c_{ij}=b_{ij}+\frac12(a_jx_i-a_ix_j)$. We have the following
Lemma \blem\label{linear-lem2} For $1\le i,j\le n$, $c_{ij}$ are
linear functions in $\{x_{\cdot},y_{\cdot\cdot}\}$.\nlem \bprf
 With the relation between $b_{ij}$ with $c_{ij}$ and the fact
 $X_ia_i=0$, we have
 For $i<j$,
\bequ\label{linear-equ}\aligned X_ib_{ij}=-a_j&\Longleftrightarrow
\frac12a_j=\frac12x_iX_ia_j-X_ic_{ij},\\
X_jb_{ij}=a_i&\Longleftrightarrow
\frac12a_i=\frac12x_jX_ja_i+X_jc_{ij},\\
X_kb_{ij}=0&\Longleftrightarrow
X_kc_{ij}=\frac12(x_iX_ka_j-x_jX_ka_i),\ k\neq i,j.
 \endaligned\nequ
Using $[X_i,X_j]=Y_{ij}$, the expression (\ref{expression-a}) and
(\ref{linear-equ}), through computation, we have,  for $1\le i<j\le
n,\ 1\le k<l\le n$, \bequ\label{linear-equ1}
Y_{ij}c_{kl}=\begin{cases} A_{il},& 1\le i<j=k<l\le n;\\
A_{jk},&1\le k<i=l<j\le n;\\
A_{ki},&1\le i,k<j=l\le n;\\
A_{lj},&1\le i=k<j,l\le n;\\
0,& \mbox{others}.
\end{cases} \nequ
 Combining
(\ref{linear-equ}) with (\ref{linear-equ1}) and the definition of
$X_i$  , through computation we have, \bequ\label{linear-equ2}
\part_ic_{kl}=\begin{cases}-\frac12B_l,\ & i=k;\\
\frac12B_k,\ &i=l;\\
0,\ &\mbox{others.}
\end{cases}
\nequ (\ref{linear-equ1}) and (\ref{linear-equ2}) yield that for
$1\le i,j\le n$, $\deg{c_{ij}}\le 1$ and we have the explicit
expression for $c_{kl}$,  for $k<l$, \bequ\label{expression-c}
c_{kl}=-\frac12B_lx_k+\frac12B_kx_l+\sum_{1\le
i<k}A_{il}y_{ik}+\sum_{l< j\le n}A_{jk}y_{lj}+\sum_{1\le
i<l}A_{ki}y_{il}+\sum_{k<j\le n}A_{lj}y_{kj}+D_{kl},\nequ where
$D_{kl}$ are constants. Combining (\ref{expression-a}), for $1\le
i<j\le n$, choose $ A_{ij}=-A_{ji}=1$ (respectively $D_{ij}=1$), the
other constants 0, we have $X=\theta_{ij}$ (respectively $Y_{ij}$).
And for $1\le i\le n$, choosing $B_i=1$ and the other constants 0,
we have $X=\hat{X}_i$.

we complete the proof.

\end{proof} \begin{proof}[Proof of Proposition \ref{linear-prop}] By the above
Lemmas, we easily complete the proof. In the concrete case of
$n=2,3$, from the equations (\ref{expression-a}) and
(\ref{expression-c}), we can easily conclude.

\end{proof}

\begin{definition}\label{radial fun} A $C^2$ function $f$:
$N_{n,2}\to\mathbb{R}$, is called radial if it satisfies  $\theta_{ij}f=0$, for all
$1\le i<j\le n$. \end{definition}

\begin{remarks}
  \begin{description}

 \item{(i).}  By the Lie relations (\ref{rotation-Lie}), $f$ is radial if and only if  $\theta_{1k}f=0$, for all $1<k\le n$. Clearly, constant functions are radial.

 \item{(ii).} If both $f,g$ are radial,  so do $k_1f\pm k_2g$, $f\cdot
g$, for $k_1,k_2\in \rr$.
\item{(iii).} In particular the heat kernel $(p_t)_{t\ge0}$ is radial. The reason is that for any function $f$, $1< k\le n$, $\theta_{1k} f(0)=0$ and
$\{\theta_{1k}\}_{1< k\le n}$ commute with $\Delta$, whence they commute
with the semigroup $P_t=e^{t\Delta}$. Hence, for any function $f$, one
has $P_{t}\theta_{1k} f=0$, which, taking the adjoint of $\theta_{1k}$
under the Lebesgue measure, which is $-\theta_{1k}$, shows that for the density $p_{t}$ of the
heat kernel at the origin $o$, one has $\theta_{1k} p_{t}=0$. This explains why
any information about the radial functions in turns give information
on the heat kernel itself.

\item{(iv).} In the Ph.D thesis of T. Melcher \cite{thesis}, the radial definition $f$ is defined by $f(\vec{x},\vec{y}_{\cdot\cdot})=g(|\vec{x}|, \vec{y}_{\cdot\cdot})$ for some smooth enough $g$. One defect in this definition is that the heat kernel $p_t$ is not radial. To some extent, our definition of radial functions is more reasonable.
\end{description}
\nrmk

\section{$\Gamma_2$ curvature}
In this section, we will prove the associated $\Gamma_2$ curvature is positive on $N_{n,2}$. It generalizes the same property for the three Brownian motion model $N_{3,2}$ (c.f. \cite{Qian2},  Proposition 3.1.). Up to the author's knowledge,  the method adopted in \cite{Qian2} is not adapted easily in our setting,  since it is not easy to express the solutions  to the associated $\frac{n(n-1)}{2}$ equations regularly, even in the case of $n=4$. To say nothing of proving the nonnegative property of $\Gamma_2$ curvature. Either it is hard to find out the certain parameter variables on which the radial functions depend for the case of $n>3$ (We remark here that in the case of $n=3$, radial functions depend on the norm of $\vec{x}$, $\vec{y}$, and their intersection angle $\<\vec{x},\vec{y}\>$). Inspired  by the ad hoc methods adopted in Section 2.9.2 in the thesis of T. Melcher, c.f. \cite{thesis}, we will prove $L^1$ heat kernel inequality for radial functions $f$, and hence the nonnegative property for $\Gamma_2$ curvature holds along the radial directions.

For simplification, denote the  following two gradient operators
$$\nabla f:=(X_1f,X_2f,\cdots,X_nf), \hat{\nabla} f:=(\hat{X}_1f,\hat{X}_2f,\cdots,\hat{X}_nf).$$

Let us first to the following key Lemma.

\begin{lemma}\label{equivgradient}
For any radial function $f$, we have
\begin{equation}\label{equal}
\sum_{i=1}^n(X_if)^2=\sum_{i=1}^n(\hat{X}_if)^2.
\end{equation}
\end{lemma}
\begin{proof}
Recall that
$$
X_if=\part_if+\frac12\left(\sum_{k<i}x_k\hat{\part}_{ki}f-\sum_{k>i}x_k\hat{\part}_{ik}f\right),
$$
and $$ \hat{X}_{i}f=\part_if-\frac12\left(\sum_{k<i}x_k\hat{\part}_{ki}f-\sum_{k>i}x_k\hat{\part}_{ik}f\right).
$$
Hence
$$
\sum_{i=1}^n(X_if)^2=\star+\sum_{i=1}^n\partial_if\left(\sum_{k<i}x_k\hat{\part}_{ki}f-\sum_{k>i}x_k\hat{\part}_{ik}f\right),
$$
and
$$
\sum_{i=1}^n(\hat{X}_if)^2=\star-\sum_{i=1}^n\partial_if\left(\sum_{k<i}x_k\hat{\part}_{ki}f-\sum_{k>i}x_k\hat{\part}_{ik}f\right),
$$
where $\star$ is sum of square of $\partial_i f$ and  $\frac12\left(\sum_{k<i}x_k\hat{\part}_{ki}f-\sum_{k>i}x_k\hat{\part}_{ik}f\right)$.
Thus to proof the desired result, we only need to prove
\begin{equation}\label{zero}
I:=\sum_{i=1}^n\partial_if\left(\sum_{k<i}x_k\hat{\part}_{ki}f-\sum_{k>i}x_k\hat{\part}_{ik}f\right)=0.
\end{equation}
Notice that
\begin{align*}
I&=\sum_{i=1}^n\sum_{k=1}^{i-1}x_k\partial_if\hat{\partial}_{ki}f-\sum_{i=1}^n\sum_{k=i+1}^nx_k\partial_if\hat{\partial}_{ik}f\\
&=\sum_{k=1}^n\sum_{i=k+1}^{n}x_k\partial_if\hat{\partial}_{ki}f-\sum_{i=1}^n\sum_{k=i+1}^nx_k\partial_if\hat{\partial}_{ik}f\\
&\stackrel{(1)}{=}\sum_{i=1}^n\sum_{k=i+1}^{n}x_i\partial_kf\hat{\partial}_{ik}f-\sum_{i=1}^n\sum_{k=i+1}^nx_k\partial_if\hat{\partial}_{ik}f\\
&=\sum_{i=1}^n\sum_{k=i+1}^{n}\left(x_i\partial_kf-x_k\partial_if\right)\hat{\partial}_{ik}f,
\end{align*}
where equality $(1)$ follows from the exchange between $i$ and $k$ in the first term. Since $f$ is radial ($\theta_{ij}f=0$), by rotation vectors $\theta_{ij}$ defined in (\ref{rotation}), we have
\begin{align*}
I&=\sum_{i=1}^n\sum_{j=i+1}^{n}\left(x_i\partial_jf-x_j\partial_if\right)\hat{\partial}_{ij}f\\
&=\sum_{i=1}^n\sum_{j=i+1}^{n}\hat{\partial}_{ij}f\Biggr(\sum_{k=1 }^{i-1}(y_{kj}\hat{\partial}_{ki}f-y_{ki}\hat{\partial}_{kj}f)+\sum_{k=i+1}^{j-1}(y_{ik}\hat{\partial}_{kj}f-y_{kj}\hat{\partial}_{ik}f)\\
&\hskip 100pt +\sum_{k=j+1}^n(y_{jk}\hat{\partial}_{ik}f-y_{ik}\hat{\partial}_{jk}f)\Biggr)\\
&=\sum_{i=1}^n\sum_{j=i+1}^{n}\sum_{k=1 }^{i-1}y_{kj}\hat{\partial}_{ki}f\hat{\partial}_{ij}f-\sum_{i=1}^n\sum_{j=i+1}^{n}\sum_{k=1 }^{i-1}y_{ki}\hat{\partial}_{kj}f\hat{\partial}_{ij}f\\
&+\sum_{i=1}^n\sum_{j=i+1}^{n}\sum_{k=i+1}^{j-1}y_{ik}\hat{\partial}_{kj}f
\hat{\partial}_{ij}f-\sum_{i=1}^n\sum_{j=i+1}^{n}\sum_{k=i+1}^{j-1}y_{kj}
\hat{\partial}_{ik}f\hat{\partial}_{ij}f\\
&+\sum_{i=1}^n\sum_{j=i+1}^{n}\sum_{k=j+1}^ny_{jk}\hat{\partial}_{ik}f\hat{\partial}_{ij}f
-\sum_{i=1}^n\sum_{j=i+1}^{n}\sum_{k=j+1}^ny_{ik}\hat{\partial}_{jk}f\hat{\partial}_{ij}f\\
&:=I_1+I_2+I_3+I_4+I_5+I_6.
\end{align*}

By change the order of summation, we have
$$
I_1=-I_6, I_2=-I_3, I_4=-I_5.
$$

Thus we finish the proof.
\end{proof}
\begin{remarks}
The relation (\ref{equal}) holds for any functions, which satisfies (\ref{zero}).  We would like to recommend  the readers to the Proposition 2.28  in T. Melcher 's Ph. D thesis \cite{thesis} for other class of functions satisfying (\ref{equal}).
\end{remarks}
Now let us statement the $L^1$ heat kernel inequality for the radial functions, where the right invariant vector fields play an essential role.

\begin{proposition}\label{L1} For any radial function $f\in C_c^{\infty}(N_{n,2})$, we have, for any $t\ge0$,
$$
|\nabla P_tf|\le P_t(|\nabla f|).
$$
\end{proposition}
\begin{proof}
Recall that for any function $h$, at the origin $o\in N_{n,2}$, we have $\nabla h=\hat{\nabla }h$. It follow, for radial function $f\in C_c^{\infty}(N_{n,2})$,
\begin{align*}
|\nabla P_tf|(o)&=|P_t\hat{\nabla} f|(o)\\
&\le P_t(|\hat{\nabla} f|)(o) \\
&= P_t(|\nabla f|)(o),
\end{align*}
where the last equality follows from the Lemma \ref{equivgradient}.
Thus  $$
|\nabla P_tf|(g) \le P_t(|\nabla f|)(g)
$$
holds for any  $g\in N_{n,2}$ by translation invariance.
\end{proof}
\begin{remarks}
The above Proposition can be compared with the Proposition 2.28 in \cite{thesis}.
\end{remarks}

 As a consequence, we have the following
\begin{theorem}\label{positive} For any compactly supported smooth,
radial function $f$, for any $t\ge0$, $g\in N_{n,2}$,
\bdes
\item{(i) Positive curvature property.} $\Gamma_2(f,f)\ge0$.
\item{(ii) LSI inequality.}  $P_t(f\log f)(g)-P_t(f)\log P_t(f)(g)\le
tP_t\left(\frac{\Gamma(f,f)}{f}\right)(g).$
\item{(iii) Isoperimetric inequality.}  $ P_t(|f-P_t(f)(g)|)(g)\le 4\sqrt{t}P_t(\Gamma(f)^{\frac12})(g).$
\ndes \end{theorem}
\begin{proof}
By Proposition \ref{L1}, $(i)$ follows from  Proposition A,  $(ii)$ and $(iii)$  follow from  Theorem 6.1 and Theorem 6.2 in \cite{BBBC}.
\end{proof}
\section{Gradient bounds for the heat kernels}
As done in \cite{BBBQ}, we have the following Li-Yau type inequality
holds.

 \begin{proposition}\label{prop-LY} There exist positive constants
$C_1,C_2,C_3$ (dependent on $n$) such that for any positive function $f$, if $u=\log
P_tf$, we have
$$
\part_t u\ge C_1\Gamma(u)+C_2t\sum_{1\le i<j\le n}|Y_{ij}u|^2-\frac{C_3}{t}.
$$
 \end{proposition}
\begin{proof} Since the proof  closely follows \cite{BBBQ}, we skip the proof.  We would like recommend   the readers'  to \cite{BBBQ} and the interesting paper \cite{BG}.
\end{proof}
As a consequence, we have the following Harnack inequality: There
exist positive constants $A_1$, $A_2$(dependent on $n$, see \cite{BG} for exact expression for $A_1$, $A_2$),  for $t_2>t_1>0$, and
$g_1,g_2\in N_{n,2}$, \bequ\label{nBM-harnack}
\frac{p_{t_1}(g_1)}{p_{t_2}(g_2)}\le
\left(\frac{t_2}{t_1}\right)^{A_1}e^{A_2\frac{d^2(g_1,g_2)}{t_2-t_1}}.
\nequ

 Let us state the first result of the gradient estimate for the heat kernel.
\begin{proposition}\label{nBM-gradient10} There exists a constant $C>0$(dependent on $n$) such that
for $t>0$, $g=(\vec{x},\vec{y})\in N_{n,2}$,
$$
\sqrt{\Gamma(\log p_t)(g)}\le C
\left(\frac{d(g)}{t}+\frac{1}{\sqrt{t}}\right),
$$
where $p_t(g)$ denotes the density of $P_t$ at $o$ and $d(g)$
denotes the Carnot-Carath\'eodory distance between $o$ and $g$.
\end{proposition} \begin{proof}
Following \cite{BB} as in \cite{Qian2}, for $0<s<t$, let
$\Phi(s)=P_s\big(p_{t-s}\log p_{t-s}\big)$, we have
$$ \Phi'(s)=P_s\big(p_{t-s}\Gamma(\log p_{t-s})\big),\
\Phi''(s)=2P_s\big(p_{t-s}\Gamma_2(\log p_{t-s})\big).
$$
By Theorem \ref{positive}, $\Phi''$ is positive, whence
$\Phi'$ is non-desceasing, thus $$\int_0^{\frac{t}2}\Phi'(s)ds\ge
\frac{t}2\Phi'(0).
$$
That is
$$
p_t\Gamma(\log p_t)\le \frac{2}{t}\big(P_{t/2}(p_{t/2}\log
p_{t/2})-p_t\log p_t\big).
$$
The right hand side can be bounded by applying the above Harnack
inequality (\ref{nBM-harnack}) and the basic fact $p_{t/2}(g)\le
p_{t/2}(o)$, for all $g\in N_{n,2}$. We have
$$
\sqrt{\Gamma(\log p_t)(g)}\le
C\left(\frac{d(g)}{t}+\frac{1}{\sqrt{t}}\right).
$$
\end{proof}

\begin{proposition}\label{smallvalue}
For $g=(\vec{x},\vec{y})\in N_{n,2}$ satisfying $d(g)\le 1$, there exists a positive constant $C$ (dependent on $n$), such that
$$
\sqrt{\Gamma( p)(g)}\le Cd(g).
$$
\end{proposition}
\begin{proof}
Recall that we have the precise expression of  the heat kernel, see (\ref{express0}). To estimate $\Gamma( p)(g)$, denote the orthogonal matrix $\Omega=(\omega_{ij})_{1\le i,j\le n}$, which appear in the $\varphi_j(A,\vec{x})$ in (\ref{express0}), we have


$$
|\varphi_j(A,\vec{x})|\le \frac{P_{2j-1}}{2}(\sinh\frac{P_{2j-1}}{2})^{-1}
$$
$$
|\partial_{i}\varphi_j(A,\vec{x})|\le (\frac{P_{2j-1}}{2})^2(\sinh\frac{P_{2j-1}}{2})^{-1}\coth\frac{P_{2j-1}}{2}(\omega_{i,2j-1}^2+\omega_{i,2j}^2)|x_i|.
$$

It yields, for $1\le i\le n$,
\begin{align}
\partial_{i}p(\vec{x},\vec{y})&\le (2\pi)^{-\frac{n(n+2)}{2}}\int_{\rr^{\frac{n(n-1)}2}}
\prod_{j=1}^{[\frac{n}{2}]}\frac{P_{2j-1}}{2}(\sinh\frac{P_{2j-1}}{2})^{-1}
\sum_{j=1}^{[\frac{n}2]}\frac{P_{2j-1}}{2}\coth\frac{P_{2j-1}}{2}(\omega_{i,2j-1}^2+\omega_{i,2j}^2)|x_i|\prod_{k<l}d\alpha_{kl}\nonumber\\
&\le (2\pi)^{-\frac{n(n+2)}{2}}|x_i|\int_{\rr^{\frac{n(n-1)}2}} \frac{P_{1}}{2}\coth\frac{P_{2[\frac{n}2]-1}}{2}
\prod_{j=1}^{[\frac{n}{2}]}\frac{P_{2j-1}}{2}(\sinh\frac{P_{2j-1}}{2})^{-1}
\prod_{k<l}d\alpha_{kl}\nonumber
\end{align}
where we  use the fact that for $1\le i\le n$, $\sum_{1\le j\le 2[\frac{n}{2}]-1}\omega_{ij}^2\le 1$, which is the consequence of the fact that $\Omega $ is orthogonal matrix.
Also we have for $k'<l'$,
$$
\hat{\partial}_{k'l'}p(\vec{x},\vec{y})\le (2\pi)^{-\frac{n(n+2)}2}\int_{\rr^{\frac{n(n-1)}2}}|\alpha_{k'l'}|\prod_{j=1}^{[\frac{n}{2}]}\frac{P_{2j-1}}{2}(\sinh\frac{P_{2j-1}}{2})^{-1}\prod_{k<l}d\alpha_{kl}.
$$
It follows,
\begin{equation}\label{gradient1}
\sum_{i=1}^n |X_ip|\le (2\pi)^{-\frac{n(n+2)}{2}}|\vec{x}|(W_1+W_2),
\end{equation}
where
$$W_1=\int_{\rr^{\frac{n(n-1)}2}} \frac{P_{1}}{2}\coth\frac{P_{2[\frac{n}2]-1}}{2}
\prod_{j=1}^{[\frac{n}{2}]}\frac{P_{2j-1}}{2}(\sinh\frac{P_{2j-1}}{2})^{-1}
\prod_{k<l}d\alpha_{kl},
$$
and
$$
W_2=\int_{\rr^{\frac{n(n-1)}2}}\sum_{k<l}|\alpha_{kl'}|\prod_{j=1}^{[\frac{n}{2}]}\frac{P_{2j-1}}{2}(\sinh\frac{P_{2j-1}}{2})^{-1}\prod_{k<l}d\alpha_{kl}
$$
with the restriction $$
\alpha:=\sum_{k<l}\alpha^2_{kl}=\sum_{j=1}^{[\frac{n}2]}P_{2j-1}^2,
$$
which follows from the fact that both sides are the  half of the trace of $-A^2=A^tA$.
 Note that for $1\le j\le [\frac{n}2] $, $P_{2j-1}\le \sqrt{\alpha}$ and $(\sinh x)^{-1}\le 4e^{-x}$ for $|x|\ge\frac12$, we have for positive constants $C_1,C_2$,
$$
W_1\le C_1 vol (B_1(0))+C_2\int_{B_1^c(0)}\alpha^{\frac{[\frac{n}2]+1}2}e^{-\sqrt{\alpha}}\prod_{k<l}d\alpha_{kl}
$$
which is obviously bounded. Similarly, we have $W_2$ is bounded.

 Combining with (\ref{gradient1}) and  the fact $|x|\le d(g)\le 1$ (see \cite{VSC}), we have
$$
\sqrt{\Gamma(p)(g)}\le C_1|x|(W_1+W_2)\le C_2|x|\le Cd(g).
$$

\end{proof}

Here is an analogue result of Theorem B in the case $N_{n,2}$.

\begin{proposition}\label{nBM-gradient1} There exists a constant $C>0$ (dependent on $n$) such that
for $t>0$, $g=(\vec{x},\vec{y})\in N_{n,2}$,
$$
\sqrt{\Gamma(\log p_t)(g)}\le
\frac{C d(g)}{t}.
$$
\end{proposition}
\begin{proof}
Taking $t=1$ in Proposition \ref{nBM-gradient10}, we have
$$
\sqrt{\Gamma(\log p)(g)}\le C\left(d(g)+1\right).
$$
If $d(g)\ge1$, it is trivial to get the desired result from the above gradient estimate. For the case $d(g)\le 1$, note that the heat kernel is bounded below by a positive constant (see \cite{VSC}), combining with Proposition \ref{smallvalue}, we have  for some positive constant $C$ (dependent on $n$),
$$
\Gamma(\log p)(g)\le Cd(g), \ \ g\in N_{n,2}.
$$
The desired result follows by the time scaling property
(\ref{scale}).
\end{proof}

\begin{remark}
In a forthcoming paper, we shall study the gradient estimates for the heat kernels of the sub-elliptic operators, which satisfy the generalized curvature dimension inequalities $CD(\rho_1,\rho_2,k,d)$ introduced by F. Baudoin and N. Garofalo in \cite{BG}.
\end{remark}

\vskip 24pt {\bf Acknowledgement}: The author would like to express
sincere thanks to Prof. H. Q. Li (Fudan University, China) for his many helpful discussion, especially for the improvement of Proposition \ref{nBM-gradient1}, also to Prof. D. Bakry (Toulouse University III, France) and F. Baudoin (Purdue University, U. S. A.)for their interest. He is greatly indebted to Prof. L. M. Wu (Clermont-Ferrand University II, France), Prof. D. Bakry  and Prof. X. D. Li (Institute of Applied Mathematics, Academia Sinica, China) for their constant encouragement and support. He also acknowledges the financial support from China (and Shanghai) Postdoctoral Scientific Program No. 20110490667 (No. 11R21412200) and from  National Science Funds (and Tianyuan fund  for Mathematics) of China No. 11171070 (No. 11126345).

\end{document}